\documentclass[11pt]{amsart}

\usepackage[margin=1in]{geometry}
\usepackage{amsmath,amssymb,amsthm,mathtools}
\usepackage{enumitem}
\usepackage{url}
\usepackage{hyperref}
\usepackage{array}

\numberwithin{equation}{section}

\newtheorem{theorem}{Theorem}[section]
\newtheorem{proposition}[theorem]{Proposition}
\newtheorem{lemma}[theorem]{Lemma}
\newtheorem{corollary}[theorem]{Corollary}

\theoremstyle{definition}

\newtheorem{remark}[theorem]{Remark}

\DeclareMathOperator{\res}{res}

\DeclareMathOperator{\cd}{cd}

\newcommand{\Fp}{\mathbb F_2}
\newcommand{\cts}{\mathrm{cts}}
\newcommand{\ab}{\mathrm{ab}}

\title[Degree-one cohomology of Metacyclic $2$-Groups]
{Degree-one cohomology of metacyclic $2$-groups: a low-bit dichotomy}

\author{Marina Palaisti}

\date{October 2026}



\begin{document}

	\begin{abstract}
		We study the subalgebra generated by degree-one mod-$2$ cohomology for metacyclic $2$-groups $G=\langle a,b\mid a^{2^\alpha}=1,\ b^{2^m}=a^c,\ b^{-1}ab=a^r\rangle$ with unique normal forms, where $\alpha,m\ge2$ and $r\equiv-1\pmod4$. Presentation compatibility forces $c=0$ or $2^{\alpha-1}$ and $G^{\mathrm{ab}}\cong C_2\times C_{2^m}$. If $c\equiv0\pmod4$, the degree-one-generated algebra is $\Fp[x,y]/(y^2,x^2+xy)$, a four-dimensional algebra shared by the four nonmodular types of Aaboun--Zekhnini. The unique exceptional case $\alpha=2$, $c=2$ is modular and has infinite-dimensional degree-one algebra $\Fp[x,y]/(y^2,x^2y)$. The regular algebra is PBW, Koszul, universally Koszul, and strongly Koszul, whereas the modular algebra is not Koszul. We realize the regular algebra by inflation from the common quotient $C_4\rtimes_{-1}C_4$ and identify its quadratic dual with the characteristic-two Jordan plane. 
		For the inversion tower $C_4\rtimes_{-1}C_{2^m}$, we give uniform presentations of the full mod-$2$ cohomology rings and determine the kernels and images of inflation in every degree. Their direct limit yields the continuous cohomology ring of $C_4\rtimes_{-1}\mathbb Z_2$. We also record an arithmetic consequence for Galois groups arising over cyclotomic $\mathbb Z_2$-extensions of real quadratic fields.
	\end{abstract}

	\maketitle
	
	\section{Introduction}\label{sec:introduction}
	The mod-$2$ cohomology of metacyclic $2$-groups has been studied at the level of the full cohomology ring. Diethelm treated nonabelian split metacyclic $p$-groups \cite{diethelm}, Rusin computed the mod-$2$ rings of metacyclic $2$-groups together with their Steenrod operations \cite{rusin}, and Huebschmann subsequently treated the mod-$p$ cohomology of general metacyclic groups \cite{huebschmann}. These calculations reveal substantial higher-degree structure even for groups with relatively simple presentations. For standard background on finite-group cohomology, see \cite{adem-milgram,benson,brown}. In particular, the theorem of Chebolu--Min\'a\v{c}--Okay--Schultz--Ure \cite[Theorem~1.1]{chebolu-minac-okay-schultz-ure-benson} implies that the full mod-$2$ cohomology ring of a nonabelian finite $2$-group is not generated by degree-one classes.
	For a graded $\Fp$-algebra $R$, write $\langle R_1\rangle$ for the subalgebra generated by its degree-one component, and set \[\operatorname{Dec}(H^*(G,\Fp)):=\langle H^1(G,\Fp)\rangle.\] This subalgebra records the cup products of degree-one characters and provides a natural object on which to study how the presentation of a metacyclic group controls its cohomology. The relation between degree-one cohomology, Koszul algebras, and quadratic duality has also been investigated in Galois cohomology and pro-$p$ group theory \cite{minac-pasini-quadrelli-tan,quadrelli-few-relations}.
	Aaboun and Zekhnini classify metacyclic-nonmodular $2$-groups with abelianization of type $(2,2^m)$ into four refined presentation types \cite[Proposition~2.4]{aaboun-zekhnini}. The types differ in their conjugation and power relations, yet their degree-one-generated cohomology algebras have the same form. We explain this uniformity directly from the lowest $2$-adic information in the presentations, before distinguishing the individual types. The relevance of the congruence modulo $4$ comes from the Bockstein itself. For a degree-one mod-$2$ character, the Bockstein can be computed using a set-theoretic lift to $\mathbb Z/4\mathbb Z$, so the resulting cocycle is determined by the corresponding low-bit multiplication law. This provides a direct route from the presentation to the degree-one cup products, without first computing the full cohomology ring.
	Consider a metacyclic group with unique normal forms \[G=\langle a,b\mid a^{2^\alpha}=1,\ b^{2^m}=a^c,\ b^{-1}ab=a^r\rangle,\qquad \alpha,m\ge2,\quad 0\le c<2^\alpha,\] and assume $r\equiv-1\pmod4$. The power relation records the carry that occurs when the exponent of $b$ passes through $2^m$. The compatibility conditions force \[c=0\quad\text{or}\quad c=2^{\alpha-1},\qquad G^{\mathrm{ab}}\cong C_2\times C_{2^m}.\] Thus the carry is divisible by $4$ except when $\alpha=2$ and $c=2$. This singles out the modular boundary directly from the presentation.
	Let $x,y\in H^1(G,\Fp)$ be the characters adapted to the generators $a,b$. In the regular branch, the normal-form multiplication modulo $4$ gives the Bockstein relations \[\beta(y)=0,\qquad \beta(x)=xy\ne0.\] As $\beta(z)=z^2$ for a degree-one mod-$2$ class, these identities determine the entire degree-one-generated algebra: \[\operatorname{Dec}(H^*(G,\Fp))\cong\Fp[x,y]/(y^2,x^2+xy).\] The algebra is four-dimensional, with basis $1,x,y,xy$. At the unique modular boundary $M_m$, the degree-one algebra instead has the presentation \[\operatorname{Dec}(H^*(M_m,\Fp))\cong\Fp[x,y]/(y^2,x^2y).\] In this case the powers of $x$ remain nonzero in every degree. Thus the divisibility of the carry by $4$ distinguishes two different algebraic behaviors.
	The uniformity in the regular branch also has a group-theoretic explanation. Reducing the two generators modulo their fourth powers gives a surjection onto $C_4\rtimes_{-1}C_4$. Inflation along this quotient induces an isomorphism of degree-one-generated cohomology algebras. Hence the common algebra of the four nonmodular types comes from a single group of order $16$ through an explicit homomorphism.
	We next examine the homological properties of these two algebras. The regular algebra is PBW, Koszul, universally Koszul, and strongly Koszul; its quadratic dual is the characteristic-two Jordan plane. The modular algebra has a minimal cubic relation and is not Koszul. Consequently, the same low-bit condition that determines the degree-one products also determines Koszulity.
	The common quotient $C_4\rtimes_{-1}C_4$ is the first level of the inversion tower \[G_m=C_4\rtimes_{-1}C_{2^m},\qquad m\ge2.\] This family permits a comparison between the stable degree-one core and the higher-degree classes of the full cohomology ring. We determine $H^*(G_m,\Fp)$ at every finite level and show that inflation kills precisely the ideal generated by the periodicity class coming from the cyclic quotient, while preserving the class restricting nontrivially to $C_4$. Passing to the direct limit gives \[H^*_{\mathrm{cts}}(C_4\rtimes_{-1}\mathbb Z_2,\Fp)\cong\Fp[x,y,\omega]/(y^2,x^2+xy),\qquad |\omega|=2.\]
	Finally, an arithmetic application shows how the group-theoretic classification determines the degree-one algebra of certain Galois groups over cyclotomic $\mathbb Z_2$-extensions of real quadratic fields. Once the arithmetic hypotheses select a metacyclic type, the cohomological and Koszul conclusions follow from the presentation.

	\section{Metacyclic presentations and low-bit compatibility}\label{sec:four-types}
	
	\subsection{Compatibility and abelianization}
	
	A metacyclic group is built from two cyclic pieces: a cyclic normal subgroup and a cyclic quotient. We consider finite metacyclic $2$-groups with a cyclic normal subgroup
	\[A=\langle a\rangle\cong C_{2^\alpha}\]
	and cyclic quotient
	\[G/A\cong C_{2^m},\]
	where $\alpha,m\ge2$. Choose an element $b\in G$ whose image generates $G/A$. Since $A$ is normal, conjugation by $b$ induces an automorphism of $A$. As $A$ is cyclic, there is an odd integer $r$ such that $b^{-1}ab=a^r$. The parameter $r$ therefore describes how the quotient generator acts on the normal cyclic subgroup.
	
	Although the image of $b$ has order $2^m$ in $G/A$, the element $b$ itself need not have order $2^m$ in $G$. What is required is that $b^{2^m}$ belongs to $A$, so there is a uniquely determined integer $c$, with $0\le c<2^\alpha$, such that
	\[b^{2^m}=a^c.\]
	Thus $c$ records the possible discrepancy between the order of $b$ and the order of its image in the quotient. More concretely, when products are expressed in terms of powers of $a$ and $b$, the relation $b^{2^m}=a^c$ determines the additional factor that appears whenever the exponent of $b$ passes through $2^m$. It is therefore convenient to regard $a^c$ as the carry in the power relation.
	
	Every element of $G$ can be expressed in the form $a^ib^j$. We assume throughout that these expressions are unique for
	\[0\le i<2^\alpha,\qquad 0\le j<2^m.\]
	In particular, $|G|=2^{\alpha+m}$, and we work with the unified presentation
	\begin{equation}\label{eq:unified-presentation}G=\langle a,b\mid a^{2^\alpha}=1,\ b^{2^m}=a^c,\ b^{-1}ab=a^r\rangle.\end{equation}
	The parameters $r$ and $c$ describe different aspects of the extension: the first determines the conjugation action, while the second determines the power relation. These parameters are nevertheless constrained by the requirement that both relations hold simultaneously.
	
	Indeed, since $b^{2^m}=a^c$ lies in the abelian subgroup $A$, conjugation by $b^{2^m}$ must act trivially on $A$. At the same time, the element $a^c=b^{2^m}$ must be fixed by conjugation by $b$. These two compatibility conditions impose congruences on $r$ and $c$, which we establish before considering the refined classification of metacyclic $2$-groups.

	\begin{proposition}\label{prop:compatibility}
		For a presentation \eqref{eq:unified-presentation} with the stated normal forms, one has
		\[r^{2^m}\equiv1\pmod{2^\alpha},\qquad c(r-1)\equiv0\pmod{2^\alpha}.\]
		If, in addition, $r\equiv-1\pmod4$, then necessarily $c=0$ or $c=2^{\alpha-1}$. Under this additional hypothesis, moreover,
		\[G'=\langle a^2\rangle,\qquad G^{\mathrm{ab}}\cong C_2\times C_{2^m}.\]
	\end{proposition}

	\begin{proof}
		For the first congruence, we compare two ways in which $b^{2^m}$ acts on the normal subgroup $A$. Since $b^{2^m}=a^c\in A$ and $A$ is abelian, conjugation by $b^{2^m}$ acts trivially on $A$. In particular,
		\[(b^{2^m})^{-1}a\,b^{2^m}=a.\]
		On the other hand, after $2^m$ conjugations by $b$ and using $b^{-1}ab=a^r$, we obtain
		\[(b^{2^m})^{-1}a\,b^{2^m}=a^{r^{2^m}}.\]
		Thus $a^{r^{2^m}}=a$, and since $a$ has order $2^\alpha$, this gives the first congruence. 
		
		For the second congruence, note that
		\[a^c=b^{2^m}=b^{-1}b^{2^m}b=b^{-1}a^cb=(b^{-1}ab)^c=a^{cr},\]
		and we hence \[c(r-1)\equiv0\pmod{2^\alpha}.\]
		
		Now if $r\equiv-1\pmod4$, then $r-1 \equiv 2 \pmod 4$, so $v_2(r-1)=1$. 
		Thus the remaining factor $2^{\alpha-1}$ must divide $c$, implying $	2^{\alpha-1}\mid c$. Because $c$ has been chosen in the range $0\le c<2^\alpha$, there are only two possibilities: $c=0\;\; \text{ or }\;\; c=2^{\alpha-1}$.
		
		We next determine the commutator subgroup. Since $G/A$ is cyclic hence abelian, it is abelian, $G'\subseteq A$. Because $A$ itself is cyclic and therefore abelian, the only nontrivial commutation is $[a,b]=a^{-1}b^{-1}ab=a^{r-1}$, and so $G'=\langle a^{r-1}\rangle$. But $v_2(r-1)=1$, so $r-1=2u$ for some odd integer $u$, which then gives
		\[\langle a^{r-1}\rangle=\langle a^{2u}\rangle=\langle a^2\rangle.\]
		Therefore $G'=\langle a^2\rangle$. It remains to identify the abelianization. Since $a^2\in G'$, the image of $a$ in $G^{\ab}=G/G'$ has order at most $2$. It is not trivial, because $a\notin\langle a^2\rangle$. Thus the image of $a$ has order exactly $2$.
		
		The defining relation $b^{2^m}=a^c$ becomes $\overline b^{\,2^m}=\overline a^{\,c}$ in the abelianization. We have already shown that $c$ is either $0$ or $2^{\alpha-1}$, and in either case $c$ is even. Since $\overline a$ has order $2$, this gives $\overline a^{\,c}=1$. Hence $\overline b^{\,2^m}=1$. Thus the order of $\overline b$ divides $2^m$. In fact its order is exactly $2^m$. If $\overline b^{\,k}=1$, then $b^k\in G'=\langle a^2\rangle\subseteq A$. Projecting to $G/A\cong C_{2^m}$ shows that the image of $b^k$ is trivial, and since the image of $b$ generates $G/A$, this forces $2^m\mid k$. Therefore $\overline b$ has order $2^m$.
		
		Finally, the two cyclic factors are independent. Suppose $\overline a^{\,\varepsilon}\overline b^{\,j}=1$, $\varepsilon\in\{0,1\}$. Then $a^\varepsilon b^j\in G'\subseteq A$. Projecting to $G/A$ gives $j\equiv0\pmod{2^m}$. The relation therefore reduces to $a^\varepsilon\in G'=\langle a^2\rangle$, which forces $\varepsilon=0$. Consequently
		\[G^{\ab}\cong C_2\times C_{2^m}.\]
	\end{proof}

	The preceding proposition explains why the carry condition that will be used later is almost automatic. Under the hypothesis $r\equiv-1\pmod4$, compatibility leaves only two possible values of the carry: $c=0$ or $c=2^{\alpha-1}$. The first gives $b^{2^m}=1$, while the second gives $b^{2^m}=a^{2^{\alpha-1}}$, the unique element of order $2$ in $A$. If $\alpha \ge 3$, then $2^{\alpha-1}\equiv0\pmod4$, so both compatible values of $c$ are divisible by $4$. Thus the condition $c \equiv 0 \pmod 4$ does not need to be imposed separately once $\alpha\ge3$.
	
	There is only one exception: When $\alpha =2$, the nonzero compatible carry is $c=2$. This is the unique compatible case in which the carry is not divisible by $4$. This exceptional situation is the modular boundary treated separately below.

	\subsection{Product convention and low-bit multiplication}
	
	\begin{proposition}\label{prop:normal-form-product}
		For $0\le  j,v<2^m$,
		\[(a^ib^j)(a^ub^v)= a^{i+r^{-j}u+\delta c}b^{j+v-\delta2^m},\]
		where $\delta=0$ if $j+v<2^m$ and $\delta=1$ if $j+v\ge 2^m$.
	\end{proposition}
	
	\begin{proof}
		The relation $b^{-1}ab=a^r$ implies $b^ja^ub^{-j}=a^{r^{-j}u}$, and hence $(a^ib^j)(a^ub^v)=a^{i+r^{-j}u}b^{j+v}$. Because $0\le  j,v<2^m$, at most one use of $b^{2^m}=a^c$ is required to return to normal form, giving the stated formula. In particular, with the convention $b^{-1}ab=a^r$, the exponent is $r^{-j}$.
	\end{proof}
	
	\begin{lemma}\label{lem:low-bit-multiplication}
		Assume $r\equiv -1\pmod 4$ and either $c=0$ or $\alpha\ge 3$. If $(a^ib^j)(a^ub^v)=a^wb^t$ is in normal form, then
		\[w\equiv i+(-1)^ju\pmod 4\]
		and
		\[t\equiv j+v\pmod2.\]
	\end{lemma}
	
	\begin{proof}
		Since $r^{-1} \equiv -1 \pmod 4$, we have that $r^{-j} \equiv (-1)^j \pmod4$. By Proposition~\ref{prop:compatibility}, the carry is either $0$ or $2^{\alpha-1}$. In either case, the carry vanishes modulo $4$, yielding the first congruence. The second follows since reduction to normal form changes the exponent of $b$ only by multiples of $2^m$, which is even.
	\end{proof}
	
	\subsection{The four refined nonmodular types}

	We recall the classification of metacyclic-nonmodular $2$-groups with abelianization of type $(2,2^m)$ given by Aaboun--Zekhnini \cite[Proposition~2.4]{aaboun-zekhnini}. Their four types are distinguished by the conjugation action and the power relation. In Types~1 and~2, the action is inversion, so $r=-1$, whereas Types~3 and~4 have $r=-1+2^s$ for an integer $s$ subject to the restrictions below. The power relation is split in Types~1 and~3, where $c=0$, and nonsplit in Types~2 and~4, where $c=2^{\alpha-1}$.
	
	The four types, together with their parameter restrictions, are summarized in the following table.

	\begin{center}
		\small \renewcommand{\arraystretch}{1.35}
		\begin{tabular}{c|c|c|p{6.0cm}}
			Type & $r$ & $c$ & Parameter restrictions\\
			\hline
			1 & $-1$ & $0$ & $\alpha\ge 2$\\
			2 & $-1$ & $2^{\alpha-1}$ & $\alpha\ge 3$\\
			3 & $-1+2^s$ & $0$& Either $s=\alpha-1$ with $\alpha\ge 3$, or $2\le s\le\alpha-2$ with $\alpha\ge 4$\\
			4 & $-1+2^s$ & $2^{\alpha-1}$
			& $2\le s\le\alpha-2$, $\alpha\ge 4$
		\end{tabular}
	\end{center}
	
	Thus, the classification separates the groups according to the conjugation action and the possible values of the power parameter $c$. Notably, the case $s=\alpha-1$ occurs only in Type~3, while Types~3 and~4 both admit $2\le s\le\alpha-2$.
	
	In every row, $r \equiv- 1 \pmod 4$. Types~1 and~3 are split. In Types~2 and~4 the classification imposes $\alpha\ge 3$, so the nonsplit carry is divisible by $4$. Hence all four types lie in the regular branch of the low-bit dichotomy without any additional cohomological hypothesis.

	\subsection{The unique low-bit boundary}
	
	The only compatible nonsplit presentation with $r\equiv-1\pmod4$ and carry not divisible by $4$ is
	\[M_m=\langle a,b\mid a^4=1,\ b^{2^m}=a^2,\ b^{-1}ab=a^{-1}\rangle, \;\; m\ge 2. \]
	
	\begin{proposition}\label{prop:modular-boundary}
		The group $M_m$ is modular. More precisely, $b$ has order $2^{m+1}$ and, with $t=ab^{2^{m-1}}$, one has
		\[t^2=1, \;\; tbt^{-1}=b^{1+2^m}.\]
		Thus
		\[ M_m\cong \langle b,t\mid b^{2^{m+1}}=1,\ t^2=1,\ tbt^{-1}=b^{1+2^m}\rangle.\]\end{proposition}
	
	\begin{proof}
		The relation $b^{2^m}=a^2\ne 1$ shows that $b$ has order $2^{m+1}$. Since $2^{m-1}$ is even, $b^{2^{m-1}}$ centralizes $a$, and therefore $t^2=a^2b^{2^m}=a^2a^2=1$. 
		
		From $b^{-1}ab=a^{-1}$ we obtain $aba^{-1}=ba^{-2}=ba^2=b^{1+2^m}$. The factor $b^{2^{m-1}}$ commutes with $b$, so the same conjugation formula holds with $a$ replaced by $t$.
	\end{proof}
	
	\begin{theorem}\label{thm:modular-boundary-core}
		Let $m\ge 2$ and let
		\[M_m=\langle a,b\mid a^4=1,\ b^{2^m}=a^2,\ b^{-1}ab=a^{-1}\rangle.\]
		Define $x,y\in H^1(M_m,\mathbb F_2)$ by
		\[x(a^i b^j)=i\pmod 2,\;\; y(a^i b^j)=j\pmod 2.\]
		Then\[ \operatorname{Dec}(H^*(M_m,\mathbb F_2)) \cong \frac{\mathbb F_2[x,y]}{(y^2,x^2y)}, \;\;| x| = | y| = 1.\]
		In particular, $xy\ne0$, $x^n\ne0$ for every $n\ge1$, and $y^2=x^2y=0$. The algebra has $\mathbb F_2$-basis $\{x^n:n\ge0\}\ \cup\ \{y,xy\}$ and Hilbert series $\displaystyle \frac{1}{1-t}+t+t^2$.
	\end{theorem}
	
	\begin{proof}
		By Proposition~2.4, with $t=ab^{2^{m-1}}$ we have
		\[M_m=K\rtimes T,\;\; K=\langle b\rangle\cong C_{2^{m+1}},\;\;T= \langle t\rangle\cong C_2,\]where conjugation by $t$ acts on $K$ as $\sigma(b)=b^{1+2^m}$. Since $m\ge2$,
		\[x(t) =1,\;\; y(t)=0,\;\; x(b)=0,\;\; y(b)= 1.\]
		Thus $x$ restricts to the generator of $H^1(T,\mathbb F_2)$, while $y$ restricts to the generator of $H^1(K,\mathbb F_2)$. Consequently, $x^n\ne0$ for every $n\ge1$, by restriction to $T$.
		
		Moreover, $y$ lifts to a homomorphism $M_m\to\mathbb Z/4\mathbb Z$ sending $a$ to $0$ and $b$ to $1$: the defining relations are preserved because $m\ge2$. Therefore the mod-$2$ Bockstein of $y$ vanishes, and $y^2=0$.

		The restriction to $T$ shows that the powers of $x$ persist, while the lift of $y$ gives $y^2=0$. To determine the remaining products, we first compute the dimension of $H^2(M_m,\Fp)$. This will constrain the differentials in the mod-$2$ Lyndon--Hochschild--Serre spectral sequence and allow us to identify the first relation involving both generators.
		Consider the integral homology Lyndon--Hochschild--Serre spectral sequence associated with $1\to K\to M_m\to T\to1$ (see \cite{lyndon,hochschild-serre}): \[E^2_{p,q}=H_p(T,H_q(K,\mathbb Z))\Longrightarrow H_{p+q}(M_m,\mathbb Z).\]
		Since $H_2(K,\mathbb Z)=H_2(T,\mathbb Z)=0$, the only potentially nonzero term in total degree two is $E^2_{1,1}=H_1(T,K)$. Under the identification $K=\mathbb Z/2^{m+1}\mathbb Z$, the action of the generator of $T$ is multiplication by $\sigma=1+2^m$. The standard cyclic homology formula \cite{brown} gives
		\[H_1(T,K)=\frac{\ker(\sigma-1)}{(\sigma+1)K}.\]
		
		Since $\sigma-1=2^m$, multiplication by $\sigma-1$ on $K=\mathbb Z/2^{m+1}\mathbb Z$ has kernel $2K$. On the other hand, $1+\sigma=2(1+2^{m-1})$, and $1+2^{m-1}$ is odd because $m\ge2$. Hence
		\[(1+\sigma)K=2K,\;\; H_1(T,K)=\frac{2K}{2K}=0.\]
		As $H_2(K,\mathbb Z)=H_2(T,\mathbb Z)=0$, all terms of total degree two on the $E_2$-page vanish. Thus $H_2(M_m,\mathbb Z)=0$. Since $M_m^{\mathrm{ab}}\cong C_2\times C_{2^m}$, the universal coefficient theorem yields
		\[\dim_{\Fp}H^2(M_m,\Fp)=2.\]
		
		We next use the mod-$2$ cohomology spectral sequence; for the general theory of extension differentials see \cite{huebschmann-differentials}. Here \[E_2^{p,q}=H^p(T,H^q(K,\mathbb F_2)) \Longrightarrow H^{p+q}(M_m,\mathbb F_2).\] The action of $T$ on $H^*(K,\mathbb F_2)$ is trivial, because $\sigma$ is an odd-power automorphism of $K$. Hence
		\[E_2^{*,*}\cong \frac{\mathbb F_2[X,\bar y,v]}{(\bar y^2)},\]
		where $|X|=(1,0)$, $|\bar y|=(0,1)$, $|v|=(0,2)$. The classes $X$ and $\bar y$ are represented by the global characters $x$ and $y$; in particular, they are permanent cycles. Their product $X\bar y$ is nonzero in $E_2^{1,1}$ and is also a permanent cycle: there are no incoming differentials, and every outgoing differential vanishes by multiplicativity. Thus
		\[xy\ne0 \;\; \text{ in } H^2(M_m,\mathbb F_2).\]
		The classes $x^2$ and $xy$ are linearly independent, since $x^2$ restricts 	nontrivially to $T$ whereas $xy$ restricts to zero. They therefore form a basis of $H^2(M_m,\mathbb F_2)$.
		
		It follows that the class $v\in E_2^{0,2}$ must be killed by a differential. Its only possible outgoing differentials are
		\[d_2:E_2^{0,2}\longrightarrow E_2^{2,1},\]
		and
		\[d_3:E_3^{0,2}\longrightarrow E_3^{3,0}.\]
		The second is impossible, since $X^3\in E_2^{3,0}$ survives to $E_\infty$, because $x^3$ restricts nontrivially to $T$. Consequently,
		\[d_2(v)=X^2\bar y\ne0,\]
		and
		\[E_\infty^{2,1}=0.\]
		
		The cohomology class $x^2y$ belongs to filtration at least two in $H^3(M_m,\mathbb F_2)$, and the vanishing of $E_\infty^{2,1}$ places it in filtration three. This last filtration step is spanned by $x^3$. Since $x^2y$ restricts to zero on $T$, while $x^3$ does not, we conclude that $x^2y=0$.
		
		We have therefore obtained a surjection
		\[\frac{\mathbb F_2[x,y]}{(y^2,x^2y)} \longrightarrow \operatorname{Dec}(H^*(M_m,\mathbb F_2)).\]
		The source has basis $\{x^n:n\ge0\}\cup\{y,xy\}$. Its images are linearly independent; indeed, in degree one this follows from the independence of $x,y$; in degree two from the independence of $x^2,xy$; and in every degree at least three from the nonvanishing of $x^n$. The surjection is therefore an isomorphism.
	\end{proof}

	The modular boundary shows how the power relation can change the degree-one algebra without changing the dimension of $H^1$. The classes $x^n$ persist in every degree, in contrast with the finite-dimensional algebra that occurs when the carry is divisible by $4$. We now turn to that regular branch and derive its relations directly from the normal-form multiplication law.

	\section{The degree-one cohomology dichotomy}\label{sec:uniform-core}

	We first consider the regular branch, where the power-relation carry is divisible by $4$. In this case, the degree-one relations can be obtained directly from the Bockstein, using the normal-form multiplication law established in Section~\ref{sec:four-types}. The exceptional modular branch has already been determined in Theorem~\ref{thm:modular-boundary-core}, and the two calculations will be combined in the main dichotomy below.
	
	Assume that \eqref{eq:unified-presentation} has unique normal forms, $r\equiv-1\pmod4$, and either $c=0$ or $\alpha\ge3$. Proposition~\ref{prop:compatibility} gives
	\[G^{\ab}\cong C_2\times C_{2^m}.\]
	
	Define
	\[x(a^ib^j)=i\pmod2, \;\; y(a^ib^j)=j\pmod2.\]
	Lemma~\ref{lem:low-bit-multiplication} shows that $x$ and $y$ are homomorphisms. Their values on $a$ and $b$ show that they are independent, and Proposition~\ref{prop:compatibility} shows that they form a basis of $H^1(G,\Fp)$.
	
	We use the Bockstein associated to
	\[0\longrightarrow\Fp \overset{\cdot2}{\longrightarrow} \mathbb Z/4\mathbb Z \longrightarrow\Fp \longrightarrow0, \]
	for which $\beta(u)=u^2$ on $H^1(-,\Fp)$.
	
	\begin{lemma}\label{lem:y-lifts}
		The character $y$ lifts to $\mathbb Z/4\mathbb Z$, and hence $y^2=\beta(y)=0$.
	\end{lemma}
	
	\begin{proof}
		Set $\widetilde y(a)=0$ and $\widetilde y(b)=1\pmod4$. Since $m\ge 2$, the left side of $b^{2^m}=a^c$ maps to $0$ modulo $4$, and the right side maps to $0$ because $\widetilde y(a)=0$. The remaining relations are automatic.
	\end{proof}
	
	\begin{lemma}\label{lem:x-no-lift}
		The character $x$ does not lift to $\mathbb Z/4\mathbb Z$. Consequently
		\[x^2=\beta(x)\ne 0.\]
	\end{lemma}
	
	\begin{proof}
		If $\widetilde x$ were a lift, then $\widetilde x(a)$ would be odd. Applying $\widetilde x$ to $b^{-1}ab=a^r$ gives
		\[\widetilde x(a)=r\widetilde x(a)\pmod4.\]
		Since $r\equiv-1\pmod4$, this would imply $2\widetilde x(a)=0\pmod4$, impossible for an odd element.
	\end{proof}

	The failure of $x$ to lift establishes that its Bockstein is nonzero, but does not yet identify the resulting degree-two class. To determine it, we use the multiplication law modulo $4$ to compute an explicit representative of $\beta(x)$.

	\begin{lemma}\label{lem:bockstein-identity}
		One has $\beta(x)=xy$.
	\end{lemma}
	
	\begin{proof}
		Use the set-theoretic lift $\widehat x(a^i b^j)=i \pmod4$. For $g=a^i b^j$ and $h=a^ub^v$, write $gh=a^wb^t$ in normal form. The connecting cocycle is
		\[(g,h)\longmapsto \frac{\widehat x(g)+\widehat x(h)-\widehat x(gh)}{2} \pmod 2.\] By Lemma~\ref{lem:low-bit-multiplication},
		\[w\equiv i+(-1)^j u \pmod 4.\]
		Hence the numerator is $0$ modulo $4$ when $j$ is even and $2u$ modulo $4$ when $j$ is odd. The resulting cocycle is
		\[(j\bmod2)(u\bmod2)=y(g)x(h),\]
		which represents $y\smile x$. Mod-$2$ group cohomology is graded-commutative, and in characteristic $2$ the degree-one sign disappears, so
		\[y\smile x=x\smile y=xy.\]
		Therefore $\beta(x)=xy$.
	\end{proof}
	
	\begin{theorem}\label{thm:lowbit-core}
		Let $G$ have a metacyclic normal-form presentation \eqref{eq:unified-presentation}, with $\alpha,m\ge 2$ and $r\equiv -1\pmod 4$. Then 
		\[c=0\;\; \text{ or } c=2^{\alpha-1},\]
		and the degree-one-generated cohomology algebra $\operatorname{Dec}(H^*(G,\Fp))$ is determined by the following dichotomy.
		\begin{enumerate}
			\item If $c\equiv0\pmod4$, equivalently if $c=0$ or $\alpha\ge3$, then
			\[\operatorname{Dec}(H^*(G,\Fp))\cong  \frac{\Fp[x,y]}{(y^2,x^2+xy)}.\]
			This algebra is four-dimensional, with $\Fp$-basis $\{1,x,y,xy\}$. 
			\item If $c\not\equiv0\pmod4$, then necessarily $\alpha=2$ and $c=2$. In this case $G$ is the modular group $M_m$, and
			\[\operatorname{Dec}(H^*(G,\Fp)) \cong \frac{ \Fp[x,y]}{(y^2,x^2y)}.\]
			This algebra is infinite-dimensional, with $\Fp$-basis $\{x^n:n\ge0\}\cup\{y,xy\}$ and Hilbert series $\displaystyle \frac{1}{1-t}+t+t^2$.\end{enumerate}
		In particular, the degree-one-generated cohomology algebra depends only on whether the power-relation parameter $c$ vanishes modulo $4$, and is independent of the remaining admissible presentation parameters.
	\end{theorem}
	
	\begin{proof}
		Proposition~\ref{prop:compatibility} shows that presentation compatibility forces $c=0$ or $c=2^{\alpha-1}$ and gives
		\[G^{\mathrm{ab}}\cong C_2\times C_{2^m}.\]
		Consequently, the presentation-adapted characters $x,y$ form a basis of $H^1(G,\Fp)$.
		
		Suppose first that $c\equiv0\pmod4$. By the compatibility condition, this occurs precisely when $c=0$ or $\alpha\ge3$. Lemmas~\ref{lem:y-lifts}, \ref{lem:x-no-lift}, and \ref{lem:bockstein-identity} give $y^2=0$, and $x^2=xy\ne 0$. These relations induce a surjective homomorphism of graded $\Fp$-algebras
		\[\frac{\Fp[x,y]}{(y^2,x^2+xy)} \longrightarrow \operatorname{Dec}(H^*(G,\Fp)).\]
		In the source algebra, every monomial of degree at least $3$ vanishes, since $x^3=x(xy)=x^2y=xy^2=0$. Moreover, the degree-two component is spanned by $xy$. Since $x$ and $y$ are linearly independent in $H^1(G,\Fp)$ and $xy\ne0$, the images of $1,x,y,xy$ are linearly independent. The homomorphism is therefore an isomorphism, giving the first assertion.
		
		Now suppose that $c\not\equiv0\pmod4$. Since $c=0$ or $c=2^{\alpha-1}$, this forces $\alpha=2$ and $c=2$. By Proposition~\ref{prop:modular-boundary}, the corresponding group is the modular group $M_m$. The exact calculation in Theorem~\ref{thm:modular-boundary-core} gives
		\[\operatorname{Dec}(H^*(M_m,\Fp)) \cong \frac{\Fp[x,y]}{(y^2,x^2y)}.\]
		The defining relations show that this algebra has basis $\{x^n:n\ge0\}\cup\{y,xy\}$. In particular, it is infinite-dimensional, and its Hilbert series is
		\[ \sum_{n\ge0}t^n+t+t^2= \frac{1}{1-t} +t+ t^2. \]
		
		The two cases exhaust all compatible values	of $c$. Their algebra presentations depend only on the residue of $c$ modulo $4$, establishing the asserted dichotomy.
	\end{proof}

	The uniformity in Theorem~\ref{thm:lowbit-core} suggests a common group-theoretic source for the degree-one algebra. Since the relevant multiplication law depends only on the presentation modulo $4$, we retain the fourth-power residues of the two generators. The resulting quotient makes the cohomological identification compatible with inflation.

	\begin{proposition}\label{prop:lowbit-quotient}
		Let $G$ have a metacyclic normal-form presentation \eqref{eq:unified-presentation}, with $\alpha,m\ge 2$, $r\equiv-1\pmod4$, and $c\equiv0\pmod4$. Let
		\[H=\langle A,B\mid A^4=B^4=1,\; B^{-1}AB=A^{-1}\rangle \cong C_4\rtimes_{-1}C_4. \]
		Then the assignments
		\[a\longmapsto A,\;\; b\longmapsto B\]
		define a surjective homomorphism
		\[\rho_G:G\twoheadrightarrow H\]
		with
		\[\ker(\rho_G)=\langle\!\langle a^4,b^4\rangle\!\rangle_G.\]
		Moreover, the induced inflation homomorphism restricts to an isomorphism of graded $\Fp$-algebras
		\[\rho_G^*: \operatorname{Dec}(H^*(H,\Fp)) \xrightarrow{\;\sim\;} \operatorname{Dec}(H^*(G,\Fp)).\]
		Thus every group in the regular branch has the same degree-one-generated cohomology algebra as the fixed group $H$ of order $16$, and this identification is induced by the presentation-defined quotient map.
	\end{proposition}
	
	\begin{proof}
		We first verify that the proposed assignments respect the defining relations of $G$. Since $\alpha\ge2$, the relation $a^{2^\alpha}=1$ maps to
		\[A^{2^\alpha}=1.\]Likewise, $m\ge2$ implies\[B^{2^m}=1.\] Because $c\equiv0\pmod4$, we also have $A^c=1$, so the power relation $b^{2^m}=a^c$ is preserved. Finally, since $r\equiv-1\pmod4$, \[B^{-1}AB=A^{-1}=A^r.\]
		Hence the assignments extend to a homomorphism $\rho_G:G\to H$. Its image contains both $A$ and $B$, so $\rho_G$ is surjective.
		
		Let
		\[N=\langle\!\langle a^4,b^4\rangle\!\rangle_G.\]
		Clearly $N\subseteq\ker(\rho_G)$. In the quotient $G/N$, the images of $a$ and $b$ have orders dividing $4$. The relation $b^{-1}ab=a^r$ becomes $b^{-1}ab=a^{-1}$, since $r\equiv-1\pmod4$. Moreover, the power relation becomes trivial because $c\equiv0\pmod4$. Consequently, $G/N$ is generated by the images of $a$ and $b$ subject to the relations defining $H$. These relations give a surjective homomorphism $H\twoheadrightarrow G/N$, while $\rho_G$ induces a surjective homomorphism $G/N\twoheadrightarrow H$. Since $H$ is finite, both maps are isomorphisms. Therefore
		\[G/N\cong H,\;\;\ker(\rho_G)=N.\]
		
		We now consider the induced map on cohomology. Let $x_H,y_H\in H^1(H,\Fp)$ be the characters determined by 
		\[x_H(A)=1,\;\; x_H(B)=0,\;\; y_H(A)=0,\;\; y_H(B)=1.\]
		Let $x_G,y_G\in H^1(G,\Fp)$ be the corresponding presentation-adapted characters. By construction,
		\[\rho_G^*(x_H)=x_G,\;\;\rho_G^*(y_H)=y_G.\]
		Since both pairs form bases of their respective first cohomology groups, $\rho_G^*$ induces an isomorphism
		\[H^1(H,\Fp)\xrightarrow{\;\sim\;}H^1(G,\Fp).\]
		
		By Theorem~\ref{thm:lowbit-core}, both $H$ and $G$ belong to the regular branch, and hence
		\[\operatorname{Dec}(H^*(H,\Fp)) \cong \frac{\Fp[x_H,y_H]} {(y_H^2,x_H^2+x_Hy_H)}\]
		and
		\[\operatorname{Dec}(H^*(G,\Fp)) \cong \frac{\Fp[x_G,y_G]} {(y_G^2,x_G^2+x_Gy_G)}.\]
		In particular, both algebras are four-dimensional, with respective bases
		\[\{1,x_H,y_H,x_Hy_H\} \text{ and }\{1,x_G,y_G,x_Gy_G\}.\]
		
		Naturality of the cup product gives
		\[\rho_G^*(x_Hy_H)=x_Gy_G.\] Thus $\rho_G^*$ maps the displayed basis of the first algebra onto the displayed basis of the second. It follows that its restriction to the degree-one-generated subalgebras is an isomorphism, as claimed.
	\end{proof}

	\begin{corollary}\label{cor:uniform-core}
		For every metacyclic-nonmodular $2$-group in Types~1--4 of Aaboun--Zekhnini,
		\[\operatorname{Dec}\left(H^*(G,\Fp)\right)\cong\Fp[x,y]/(y^2,x^2+xy).\]
		Thus the degree-one-generated algebra is independent of the refined type and all admissible parameters.
	\end{corollary}
	
	\begin{proof}
		Each of the four source presentations has $r\equiv-1\pmod4$, and none is the exceptional $\alpha=2$, $c=2$ boundary. Theorem~\ref{thm:lowbit-core} applies.
	\end{proof}

	\begin{remark}
		Theorem~\ref{thm:lowbit-core} isolates the low-bit mechanism behind the degree-one relations across the refined metacyclic types. The full cohomology calculations of Rusin and Huebschmann describe the corresponding ambient rings. The normal-form calculation used here makes the uniformity transparent: the congruence $r\equiv-1\pmod4$ fixes the conjugation contribution, while the compatibility condition on $c$ controls the power-relation carry modulo $4$. \end{remark}
	
	\begin{corollary}\label{cor:bockstein-table}
		In the regular branch,
		\[ \beta(y)=0, \;\; \beta(x) = xy, \;\; \beta(x+y)=xy.\]
		Among the three nonzero elements of $H^1(G,\Fp)$, exactly one has vanishing Bockstein.
	\end{corollary}

	\section{Koszulity and Quadratic Duality}\label{sec:koszul}

	The explicit algebra presentations allow us to compare homological properties in the two branches. Related enhanced-Koszul questions have been studied for right-angled Artin groups and in Galois cohomology \cite{cassella-quadrelli,minac-palaisti-pasini-tan}. The regular algebra is defined by quadratic relations, while the modular boundary has a minimal cubic relation. We first establish the enhanced Koszul properties of the regular algebra and then show that the same low-bit distinction determines Koszulity.

	We use the colon criterion for universal Koszulity as developed by Conca \cite{conca-universal} and discussed by Conca--De Negri--Rossi \cite{conca-de-negri-rossi}. For strong Koszulity we use the notion introduced by Herzog--Hibi--Restuccia \cite{herzog-hibi-restuccia}, also used in the enhanced-Koszul literature \cite{minac-palaisti-pasini-tan}.

	\begin{proposition}\label{prop:koszul-package}
		The algebra $\Fp[x,y]/(y^2,x^2+xy)$ is PBW and therefore Koszul. It is also universally Koszul and strongly Koszul.
	\end{proposition}
	
	\begin{proof}
		With degree-lexicographic order $x>y$, the two quadrics $x^2+xy, \;\; y^2$ form a Gr\"{o}bner basis: their only $S$-polynomial is $xy^3$, which reduces to zero. Thus the algebra is PBW.
		
		For universal Koszulity, set $z=x+y$. The nonzero proper ideals generated by degree-one elements are $(x),\;\;(y),\;\;(z),\;\;\mathfrak m=(x,y)$. Using
		\[x^2=xy, \;\; y^2=0, \;\; z^2=xy, \;\; xz=0, \;\; yz=xy,\]
		one obtains
		\[(0:x)=(z), \;\; (0:y)=(y), \;\; (0:z)=(x),\]
		and, whenever $\ell$ is a nonzero linear form outside the degree-one part of one of $(x),(y),(z)$,
		\[((x):\ell)=((y):\ell)=((z):\ell)=\mathfrak{m}\]
		in the applicable cases. Hence the standard colon criterion gives universal Koszulity.
		
		For strong Koszulity use the basis $\{x,z\}$. Then
		\[ \Fp[x,y]/(y^2,x^2+xy) \cong\Fp[x,z]/(xz,x^2+z^2),\]
		and
		\[ (0:x)=(z), \;\; (0:z)=(x), \;\; (x):z=(x,z), \;\; (z):x=(x,z).\]
		Every required colon is generated by a subset of the same basis.
	\end{proof}

	\begin{corollary}\label{cor:koszulity-dichotomy}
		Let $G$ have a metacyclic normal-form presentation \eqref{eq:unified-presentation}, with $\alpha,m\ge2$ and $r\equiv-1\pmod4$. Then the degree-one-generated cohomology algebra $\operatorname{Dec}(H^*(G,\Fp))$ is Koszul if and only if $c\equiv0\pmod4$.
		
		More precisely, if $c\equiv0\pmod4$, then $\operatorname{Dec}(H^*(G,\Fp))$ is PBW, Koszul, universally Koszul, and strongly Koszul. If $c\not\equiv0\pmod4$, then necessarily $\alpha=2$ and $c=2$, and the corresponding algebra is not quadratic and hence is not Koszul. In particular, all four nonmodular metacyclic types of Aaboun--Zekhnini lie in the Koszul branch.
	\end{corollary}
	
	\begin{proof}
		By Theorem~\ref{thm:lowbit-core}, the degree-one-generated cohomology algebra has precisely two possible presentations, depending on whether $c$ vanishes modulo $4$.
		
		Suppose first that $c\equiv0\pmod4$. Then
		\[\operatorname{Dec}(H^*(G,\Fp))\cong\Fp[x,y]/(y^2,x^2+xy).\]
		The preceding proposition shows that this algebra is PBW, Koszul, universally Koszul, and strongly Koszul.
		
		Now suppose that $c\not\equiv0\pmod4$. Presentation compatibility forces $\alpha=2$ and $c=2$, and Theorem~\ref{thm:lowbit-core} gives
		\[S:=\operatorname{Dec}(H^*(G,\Fp))\cong\Fp[x,y]/(y^2,x^2y).\]
		The relation $x^2y=0$ is a minimal cubic relation. Indeed, the only quadratic relation is $y^2=0$, and $x^2y$ remains nonzero in $\Fp[x,y]/(y^2)$. Consequently, the defining ideal of $S$ cannot be generated by quadratic relations, so $S$ is not quadratic. Since every connected standard graded Koszul algebra is quadratic (see \cite{priddy}), $S$ is not Koszul.
		
		The failure of universal Koszulity can also be seen directly from the colon criterion. In $S$ one has $(0:x)=(xy)$. Since $xy\ne0$ and $(xy)$ contains no nonzero element of degree one, this annihilator is not generated by linear forms. Thus $S$ fails the universal Koszul colon criterion.
		
		Finally, the PBW, universal Koszul, and strong Koszul properties each imply Koszulity. Hence none of these enhanced properties holds in the exceptional branch. This establishes the dichotomy.
	\end{proof}

	Koszulity makes the quadratic dual a natural object to consider, since it describes the Yoneda algebra of the degree-one core; see \cite{beilinson-ginzburg-soergel,polishchuk-positselski} for general treatments of quadratic duality. In the regular branch, the dual admits a familiar noncommutative presentation, connecting the cohomology calculation with the characteristic-two Jordan plane \cite{iyudu-jordan-plane}.

	\begin{proposition}\label{prop:regular-quadratic-dual}
		Let
		\[R=\Fp[x,y]/(y^2,x^2+xy),\;\; |x|=|y|=1.\]
		Then its quadratic dual is
		\[R^!\cong\frac{\Fp\langle U,V\rangle}{(VU+UV+U^2)},\;\; |U|=|V|=1.\]
		In particular, $R^!$ is the Jordan plane in characteristic two. It has an $\Fp$-basis $\{U^iV^j:i,j\ge0\}$ and Hilbert series
		\[H_{R^!}(t)=\frac{1}{(1-t)^2}.\]
		Since $R$ is Koszul, its Yoneda algebra satisfies
		\[\operatorname{Ext}_R^*(\Fp,\Fp)\cong R^!,\;\; \dim_{\Fp}\operatorname{Ext}_R^n(\Fp,\Fp)=n+1\;\;(n\ge0).\]
		Consequently, this description applies to the degree-one-generated cohomology algebra of every group in the regular branch of Theorem~\ref{thm:lowbit-core}.
	\end{proposition}
	
	\begin{proof}
		Let $W=\Fp x\oplus\Fp y$. To compute the quadratic dual, write $R$ as a quotient of the tensor algebra $T(W)$. Since $R$ is commutative, its space of quadratic relations is
		\[J=\langle y\otimes y,\ x\otimes x+x\otimes y,\ x\otimes y+y\otimes x\rangle\subseteq W\otimes W.\]
		Let $U,V$ be the dual basis of $W^*$. An element
		\[\lambda\,U\otimes U+\mu\,U\otimes V+\nu\,V\otimes U+\kappa\,V\otimes V\]
		belongs to the orthogonal complement $J^\perp$ precisely when
		\[\kappa=0,\;\; \lambda+\mu=0,\;\; \mu+\nu=0.\]
		Thus $J^\perp$ is generated by $U\otimes U+U\otimes V+V\otimes U$, and the definition of the quadratic dual gives
		\[R^!=T(W^*)/(J^\perp)\cong\Fp\langle U,V\rangle/(VU+UV+U^2).\]
		
		Equivalently, its defining relation is $VU=UV+U^2$. Let $\delta$ be the derivation of $\Fp[U]$ determined by $\delta(U)=U^2$. Then the same relation presents the Ore extension
		\[R^!\cong\Fp[U][V;\delta].\]
		The Ore-extension normal form, equivalently obtained by the reduction $VU\mapsto UV+U^2$, gives the $\Fp$-basis $\{U^iV^j:i,j\ge0\}$ (see \cite{mcconnell-robson,bergman}). Hence its degree-$n$ component has dimension $n+1$, and
		\[H_{R^!}(t)=\sum_{n\ge0}(n+1)t^n=\frac{1}{(1-t)^2}.\]
		
		On the other hand, the relations of $R$ give the basis $\{1,x,y,xy\}$, so $H_R(t)=(1+t)^2$ and
		\[H_R(t)H_{R^!}(-t)=1.\]
		By Proposition~4.1, $R$ is Koszul. The standard identification of the Yoneda algebra of a Koszul algebra with its quadratic dual therefore gives $\operatorname{Ext}_R^*(\Fp,\Fp)\cong R^!$. Finally, Theorem~\ref{thm:lowbit-core} identifies the degree-one-generated cohomology algebra of every group in the regular branch with $R$, so the same description applies throughout that family.
	\end{proof}

	\section{Cohomology and inflation in the inversion family}\label{sec:limit}

	The group $C_4\rtimes_{-1}C_4$ appearing in Proposition~\ref{prop:lowbit-quotient} is also the first level of a natural inverse system of metacyclic groups. The degree-one algebra is unchanged throughout this system, while higher-degree classes exhibit different behavior under inflation. We determine the full finite-level cohomology rings and their inflation maps in order to identify the classes that persist in the continuous limit.

	Let
	\[G_m=\langle a,b\mid a^4=1,\ b^{2^m}=1,\ b^{-1}ab=a^{-1}\rangle=C_4\rtimes_{-1}C_{2^m},\;\; m\ge 2,\]
	with quotient maps
	\[\pi_m:G_{m+1}\twoheadrightarrow G_m.\]
	Then
	\[\Gamma:=\varprojlim_m G_m\cong C_4\rtimes_{-1}\mathbb Z_2.\]

	\subsection{Finite-level cohomology and inflation}

	The mod-$2$ cohomology of $G_2=C_4\rtimes_{-1}C_4$ was computed by David Benson in Appendix~A of \cite{chebolu-minac-okay-schultz-ure-benson}, in the course of the proof of Proposition~A.1. In his notation, the group is
	\[H_2=\langle g,h\mid g^4=h^4=1,\ gh=h^{-1}g\rangle,\]
	which agrees with our $G_2$ under $g=b$ and $h=a$.
	
	Benson's calculation gives, after renaming the generators to agree with our degree-one characters,
	\[H^*(G_2,\Fp)\cong\frac{\Fp[x,y,u,w]}{(y^2,x^2+xy)},\;\; |x|=|y|=1,\;\; |u|=|w|=2.\]
	Here $u$ is inflated from the degree-two generator of the cohomology of the cyclic quotient $G_2/\langle a\rangle\cong C_4$.
	
	The second degree-two generator $w$ arises in Benson's calculation from the central extension
	\[1\longrightarrow\langle a^2\rangle\longrightarrow G_2\longrightarrow C_4\times C_2\longrightarrow1.\]
	Writing $H^*(\langle a^2\rangle,\Fp)=\Fp[\zeta]$, with $|\zeta|=1$, the associated Lyndon--Hochschild--Serre spectral sequence has
	\[d_2(\zeta)=x^2+xy.\]
	Benson shows that $\zeta^2$ survives the spectral sequence and gives rise to the degree-two generator $w$. In particular,
	\[\res_{\langle a^2\rangle}(w)=\zeta^2\ne0.\]
	Since $\langle a^2\rangle\subseteq\langle a\rangle$, it follows that
	\[\res_{\langle a\rangle}(w)\ne0.\]
	
	We take $\omega_{A,2}=w$. For each $m\ge2$, let
	\[q_m:G_m\twoheadrightarrow G_2\]
	be the quotient map fixing $a$ and reducing the exponent of $b$ modulo $4$, and define
	\[\omega_{A,m}:=q_m^*(\omega_{A,2})\in H^2(G_m,\Fp).\]
	Because $q_m$ restricts to the identity on $\langle a\rangle\cong C_4$, naturality of restriction gives
	\[\res_{\langle a\rangle}(\omega_{A,m})=\res_{\langle a\rangle}(\omega_{A,2})\ne0.\]
	Thus the degree-two classes $\omega_{A,m}$ are nonzero at every finite level.
	
	Let
	\[\omega_{B,m}\in H^2(G_m,\Fp)\]
	be the inflation of the standard degree-two class of $C_{2^m}$. Since the quotient $G_m\twoheadrightarrow C_{2^m}$ splits, this class is nonzero.
	
	\begin{lemma}\label{lem:cyclic-inflation-zero}
		The inflation map
		\[H^2(C_{2^m},\Fp)\longrightarrow H^2(C_{2^{m+1}},\Fp)\]
		is zero for every $m\ge2$.
	\end{lemma}
	
	\begin{proof}
		For
		\[1\longrightarrow C_2\longrightarrow C_{2^{m+1}}\longrightarrow C_{2^m}\longrightarrow1,\]
		the first inflation map in the five-term sequence is an isomorphism on $H^1$, because the parity character pulls back to the parity character. The following restriction map to $H^1(C_2,\Fp)$ is therefore zero, so the transgression
		\[H^1(C_2,\Fp)\longrightarrow H^2(C_{2^m},\Fp)\]
		is injective. Both groups are one-dimensional, hence the transgression is an isomorphism and the next inflation map in degree $2$ is zero.
	\end{proof}

	The preceding lemma determines what happens to the quotient-side periodicity class, while the construction of $\omega_{A,m}$ gives a compatible class restricting nontrivially to $C_4$. To describe inflation beyond degree $2$, we must identify these classes within the full cohomology ring. The following proposition provides this description and determines the kernel and image of inflation in every degree.

	\begin{proposition}\label{prop:omegaB-dies}\label{prop:full-inflation}
		For $m\ge2$, write $x_m,y_m\in H^1(G_m,\Fp)$ for the presentation-adapted characters and set
		\[u_m=\omega_{B,m},\;\; w_m=\omega_{A,m}.\]
		Then, with $|x_m|=|y_m|=1$ and $|u_m|=|w_m|=2$,
		\[H^*(G_m,\Fp)\cong\frac{\Fp[x_m,y_m,u_m,w_m]}{(y_m^2,x_m^2+x_my_m)}.\]
		Under the inflation homomorphism $\pi_m^*:H^*(G_m,\Fp)\longrightarrow H^*(G_{m+1},\Fp)$, the generators satisfy
		\[\pi_m^*(x_m)=x_{m+1},\;\;\pi_m^*(y_m)=y_{m+1},\;\;\pi_m^*(u_m)=0,\;\;\pi_m^*(w_m)=w_{m+1}.\]
		Consequently,
		\[\ker(\pi_m^*)=(u_m)\]
		and
		\[\operatorname{im}(\pi_m^*)=\frac{\Fp[x_{m+1},y_{m+1},w_{m+1}]}{(y_{m+1}^2,x_{m+1}^2+x_{m+1}y_{m+1})}.\]
		In particular, $\dim_{\Fp}H^n(G_m,\Fp)=n+1$ for every $n\ge0$. The rank of $\pi_m^*$ is $1$ in degree $0$ and $2$ in each positive degree, and its kernel has dimension $n-1$ in degree $n\ge2$.
	\end{proposition}
	
	\begin{proof}
		
		Let $A=\langle a\rangle\cong C_4$ and $Q_m=G_m/A=\langle\bar b\rangle\cong C_{2^m}$, where $\bar b=bA$. Since the quotient map has the section $\bar b\mapsto b$, we have the split extension \[1\longrightarrow A\longrightarrow G_m\longrightarrow Q_m\longrightarrow1.\] We use its Lyndon--Hochschild--Serre spectral sequence to compare the cohomology of $G_m$ with that of the two cyclic groups; for resolutions adapted to group extensions, see \cite{wall}. The generator $\bar b$ acts on $A$ by inversion, which acts trivially on both the degree-one character and the degree-two periodicity generator of $H^*(A,\Fp)$. Thus \[E_2^{p,q}=H^p(Q_m,H^q(A,\Fp))\Longrightarrow H^{p+q}(G_m,\Fp)\] and \[E_2^{*,*}\cong\frac{\Fp[\bar x,\bar y,\bar u,\bar w]}{(\bar x^2,\bar y^2)},\] where \[|\bar x|=(0,1),\quad |\bar w|=(0,2),\quad |\bar y|=(1,0),\quad |\bar u|=(2,0).\] The classes $\bar x,\bar w$ come from $H^*(A,\Fp)$, while $\bar y,\bar u$ come from $H^*(Q_m,\Fp)$.
		To determine the differentials, we identify these four generators with classes already defined on $G_m$. The classes $y_m$ and $u_m$ are inflated from $Q_m$ and represent $\bar y$ and $\bar u$. The restriction of $x_m$ to $A$ is the nonzero degree-one character, while the restriction of $w_m$ is the nonzero degree-two periodicity class. Thus $x_m,w_m$ represent $\bar x,\bar w$ under the edge maps. All four generators are permanent cycles because they are represented by classes in $H^*(G_m,\Fp)$. Since they generate the $E_2$-page multiplicatively, every differential vanishes and the spectral sequence collapses at $E_2$.

		The classes $x_m,y_m,u_m,w_m$ generate $H^*(G_m,\Fp)$, since their images generate the associated graded algebra. By Theorem~\ref{thm:lowbit-core},
		\[y_m^2=0,\qquad x_m^2=x_my_m.\]
		These relations give a surjective homomorphism
		\[\frac{\Fp[x_m,y_m,u_m,w_m]}{(y_m^2,x_m^2+x_my_m)}\longrightarrow H^*(G_m,\Fp).\]
		The source has Hilbert series
		\[\frac{(1+t)^2}{(1-t^2)^2}=\frac{1}{(1-t)^2}.\]
		The collapse of the spectral sequence gives the same Hilbert series for $H^*(G_m,\Fp)$, so the surjection is an isomorphism in every degree.
		
		The maps $\pi_m$ fix $a$ and reduce the exponent of $b$ modulo $2^m$. Consequently, the presentation-adapted characters satisfy
		\[\pi_m^*(x_m)=x_{m+1},\qquad\pi_m^*(y_m)=y_{m+1}.\]
		The identities $q_{m+1}=q_m\circ\pi_m$ and $w_m=q_m^*(\omega_{A,2})$ give
		\[\pi_m^*(w_m)=w_{m+1}.\]
		Likewise, $u_m$ is inflated from $C_{2^m}$, and naturality together with Lemma~\ref{lem:cyclic-inflation-zero} gives
		\[\pi_m^*(u_m)=0.\]
		
		Under the ring presentations just established, $\pi_m^*$ is therefore the homomorphism obtained by setting $u_m=0$ and identifying the other three generators with their counterparts at level $m+1$. It follows that its kernel is $(u_m)$ and its image is the displayed subalgebra. The latter has Hilbert series
		\[\frac{(1+t)^2}{1-t^2}=\frac{1+t}{1-t},\]
		so its degree-zero component has dimension $1$ and every positive-degree component has dimension $2$. Comparison with $1/(1-t)^2$ gives the assertions about ranks and kernel dimensions.
	\end{proof}
	
	\subsection{The continuous ring as the resulting limit}

	The finite-level calculation separates the two periodicity directions in every cohomological degree. The quotient-side class $u_m$ disappears under the next inflation map, while the subalgebra generated by $x_m,y_m,w_m$ survives unchanged. The continuous cohomology ring is therefore obtained by retaining precisely this stable part of the finite cohomology.

	For finite discrete coefficients,
	\[H^q_{\cts}(\Gamma,\Fp)\cong\varinjlim_m H^q(G_m,\Fp)\]
	with transition maps given by inflation; see \cite[Proposition~4.18]{harari} and the standard treatments of continuous cohomology in \cite{serre-galois,neukirch-schmidt-wingberg}.
	
	The continuous LHS spectral sequence for
	\[1\longrightarrow C_4\longrightarrow\Gamma\longrightarrow\mathbb Z_2\longrightarrow1\]
	has
	\[E_2^{p,q}\cong H^p_{\cts}(\mathbb Z_2,\Fp)\otimes H^q(C_4,\Fp).\]
	The inversion action is trivial on $H^*(C_4,\Fp)$, and $\cd_2(\mathbb Z_2)=1$, so the spectral sequence is concentrated in columns $0$ and $1$ and collapses. The continuous Bockstein calculation, as in Section~\ref{sec:uniform-core}, gives
	\[y^2=0,\qquad x^2=xy.\]
	A class $\omega\in H^2_{\cts}(\Gamma,\Fp)$ restricting to the nonzero periodicity class of $C_4$ has nonzero powers in every even degree. The finite-level inflation calculation identifies this class and the resulting ring structure explicitly.
	
	\begin{corollary}\label{thm:continuous-ring}
		One has
		\[H^*_{\cts}(C_4\rtimes_{-1}\mathbb Z_2,\Fp)\cong\frac{\Fp[x,y,\omega]}{(y^2,x^2+xy)},\qquad |x|=|y|=1,\quad |\omega|=2.\]
		Under the direct-limit identification, $\omega$ is represented by the compatible family $\{\omega_{A,m}\}_{m\ge2}$, while the classes $\omega_{B,m}$ represent zero in the limit.
	\end{corollary}
	
	\begin{proof}
		By Proposition~\ref{prop:full-inflation}, each finite-level ring has the presentation
		\[H^*(G_m,\Fp)\cong\frac{\Fp[x_m,y_m,u_m,w_m]}{(y_m^2,x_m^2+x_my_m)},\]
		and inflation preserves $x_m,y_m,w_m$ while sending $u_m$ to zero. Thus every class divisible by $u_m$ vanishes at the next level, whereas the subalgebras generated by $x_m,y_m,w_m$ are identified isomorphically under inflation. Taking the direct limit gives
		\[\varinjlim_m H^*(G_m,\Fp)\cong\frac{\Fp[x,y,\omega]}{(y^2,x^2+xy)},\qquad |\omega|=2.\]
		Here $\omega$ is represented by $\{w_m\}_{m\ge2}=\{\omega_{A,m}\}_{m\ge2}$, and the quotient-side classes $u_m=\omega_{B,m}$ vanish in the limit. The asserted continuous cohomology ring follows.
	\end{proof}
	
	\begin{remark}
		The same continuous LHS argument applies to $C_{2^\alpha}\rtimes_r\mathbb Z_2$ with $\alpha\ge2$ and an action $r\equiv-1\pmod4$ of $2$-power order. The inversion family considered here additionally provides explicit finite-level representatives and determines the inflation maps in every degree.
	\end{remark}

	\section{An arithmetic consequence}\label{sec:arithmetic}
	The preceding results apply whenever arithmetic conditions identify a Galois group with one of the metacyclic families considered here. In that situation, the group-theoretic classification determines its degree-one cohomology algebra and its Koszul properties. We illustrate this with the real quadratic fields studied by Chems-Eddin and El Mamry.

	Let $\eta\in\{1,2\}$ and let $q,r,s$ be distinct odd primes satisfying the residue-symbol, unit, and $2$-class-number conditions considered by Chems-Eddin and El Mamry \cite[Theorem~3.10]{chems-eddin-el-mamry}. Set $F=\mathbb Q(\sqrt{\eta qrs})$, and write $h_2(d)$ for the largest power of $2$ dividing the class number of $\mathbb Q(\sqrt d)$. Let $F_\infty$ denote the cyclotomic $\mathbb Z_2$-extension of $F$, and let $\mathcal L(F_\infty)$ be its maximal unramified pro-$2$-extension. We consider the Galois group
	\[G_{F_\infty}=\operatorname{Gal}(\mathcal L(F_\infty)/F_\infty).\]
	Theorem~3.10 of Chems-Eddin and El Mamry identifies the metacyclic structure of this group under the stated arithmetic conditions. In particular, when the quartic-residue symbols satisfy $\left(\frac{r}{s}\right)_4\ne\left(\frac{s}{r}\right)_4$, their result selects a group of Type~1 with $\alpha=2$ in the classification recalled in Section~\ref{sec:four-types}. This allows us to apply the degree-one cohomology and Koszulity results established above.
	
	\begin{corollary}\label{cor:arithmetic-core}
		Assume the hypotheses of \cite[Theorem~3.10]{chems-eddin-el-mamry}, and let $m\ge2$ be determined by $h_2(\eta qrs)=2^{m+1}$. If \[\left(\frac{r}{s}\right)_4\ne\left(\frac{s}{r}\right)_4,\] then \[G_{F_\infty}\cong C_4\rtimes_{-1}C_{2^m}\] and \[\operatorname{Dec}(H^*(G_{F_\infty},\Fp))\cong\frac{\Fp[x,y]}{(y^2,x^2+xy)}.\] In particular, this degree-one algebra is PBW, Koszul, universally Koszul, and strongly Koszul.
	\end{corollary}
	\begin{proof}
		By \cite[Theorem~3.10]{chems-eddin-el-mamry}, the quartic-residue condition identifies $G_{F_\infty}$ as the Type~1 group with $\alpha=2$ and quotient exponent $m$. The degree-one algebra follows from Theorem~\ref{thm:lowbit-core}, and its enhanced Koszul properties follow from Corollary~\ref{cor:koszulity-dichotomy}.
	\end{proof}


\begin{thebibliography}{99}
		
		\bibitem{aaboun-zekhnini}
		B. Aaboun and A. Zekhnini, On the metacyclic \(2\)-groups whose abelianizations are of type \((2,2^n)\), \(n\ge 2\), and applications, \textit{Research in Number Theory} \textbf{9} (2023), Paper No.~55. \url{https://doi.org/10.1007/s40993-023-00461-x}
		
		\bibitem{adem-milgram}
		A. Adem and R. J. Milgram, \textit{Cohomology of Finite Groups}, 2nd ed., Grundlehren der mathematischen Wissenschaften, Vol.~309, Springer, Berlin, 2004. \url{https://doi.org/10.1007/978-3-662-06280-7}
		
		\bibitem{beilinson-ginzburg-soergel}
		A. Beilinson, V. Ginzburg, and W. Soergel, Koszul duality patterns in representation theory, \textit{Journal of the American Mathematical Society} \textbf{9} (1996), no.~2, 473--527. \url{https://doi.org/10.1090/S0894-0347-96-00192-0}
		
		\bibitem{benson}
		D. J. Benson, \textit{Representations and Cohomology II: Cohomology of Groups and Modules}, Cambridge Studies in Advanced Mathematics, Vol.~31, Cambridge University Press, Cambridge, 1991. \url{https://doi.org/10.1017/CBO9780511623622}
		
		\bibitem{bergman}
		G. M. Bergman, The diamond lemma for ring theory, \textit{Advances in Mathematics} \textbf{29} (1978), no.~2, 178--218. \url{https://doi.org/10.1016/0001-8708(78)90010-5}
		
		\bibitem{brown}
		K. S. Brown, \textit{Cohomology of Groups}, Graduate Texts in Mathematics, Vol.~87, Springer-Verlag, New York, 1982. \url{https://doi.org/10.1007/978-1-4684-9327-6}
		
		\bibitem{cassella-quadrelli}
		A. Cassella and C. Quadrelli, Right-angled Artin groups and enhanced Koszul properties, \textit{Journal of Group Theory} \textbf{24} (2021), no.~1, 17--38. \url{https://doi.org/10.1515/jgth-2020-0049}
		
		\bibitem{chebolu-minac-okay-schultz-ure-benson}
		S. K. Chebolu, J. Min\'a\v{c}, C. Okay, A. Schultz, C. Ure, and D. Benson, The Bloch--Kato conjecture, decomposing fields, and generating cohomology in degree one, \textit{Journal of Number Theory} \textbf{286} (2026), 301--341. \url{https://doi.org/10.1016/j.jnt.2026.02.007}
		
		\bibitem{chems-eddin-el-mamry}
		M. M. Chems-Eddin and H. El Mamry, On the existence of the maximal unramified pro-\(2\)-extension over the cyclotomic \(\mathbb Z_2\)-extension with prescribed metacyclic Galois group, \textit{The Ramanujan Journal} \textbf{70} (2026), Article~17. \url{https://doi.org/10.1007/s11139-026-01381-x}
		
		\bibitem{conca-universal}
		A. Conca, Universally Koszul algebras, \textit{Mathematische Annalen} \textbf{317} (2000), no.~2, 329--346. \url{https://doi.org/10.1007/s002080000100}
		
		\bibitem{conca-de-negri-rossi}
		A. Conca, E. De Negri, and M. E. Rossi, Koszul algebras and regularity, in \textit{Commutative Algebra: Expository Papers Dedicated to David Eisenbud on the Occasion of his 65th Birthday}, Springer, New York, 2013, pp.~285--315. \url{https://doi.org/10.1007/978-1-4614-5292-8_8}
		
		\bibitem{diethelm}
		T. Diethelm, The mod-\(p\) cohomology rings of the nonabelian split metacyclic \(p\)-groups, \textit{Archiv der Mathematik} \textbf{44} (1985), no.~1, 29--38. \url{https://doi.org/10.1007/BF01193778}
		
		\bibitem{harari}
		D. Harari, \textit{Galois Cohomology and Class Field Theory}, Universitext, Springer, 2020. \url{https://doi.org/10.1007/978-3-030-43901-9}
		
		\bibitem{herzog-hibi-restuccia}
		J. Herzog, T. Hibi, and G. Restuccia, Strongly Koszul algebras, \textit{Mathematica Scandinavica} \textbf{86} (2000), no.~2, 161--178.  \url{https://doi.org/10.7146/math.scand.a-14287}
		
		\bibitem{hochschild-serre}
		G. Hochschild and J.-P. Serre, Cohomology of group extensions, \textit{Transactions of the American Mathematical Society} \textbf{74} (1953), no.~1, 110--134. \url{https://doi.org/10.1090/S0002-9947-1953-0052438-8}
		
		\bibitem{huebschmann-differentials}
		J. Huebschmann, Automorphisms of group extensions and differentials in the Lyndon--Hochschild--Serre spectral sequence, \textit{Journal of Algebra} \textbf{72} (1981), no.~2, 296--334. \url{https://doi.org/10.1016/0021-8693(81)90296-9}
		
		\bibitem{huebschmann}
		J. Huebschmann, The mod-\(p\) cohomology rings of metacyclic groups, \textit{Journal of Pure and Applied Algebra} \textbf{60} (1989), no.~1, 53--103. \url{https://doi.org/10.1016/0022-4049(89)90107-2}
		
		\bibitem{iyudu-jordan-plane}
		N. K. Iyudu, Representation spaces of the Jordan plane, Communications in Algebra \textbf{42} (2014), no.~8, 3507--3540. \url{https://doi.org/10.1080/00927872.2013.788184}
		
		\bibitem{lyndon}
		R. C. Lyndon, The cohomology theory of group extensions, \textit{Duke Mathematical Journal} \textbf{15} (1948), no.~1, 271--292. \url{https://doi.org/10.1215/S0012-7094-48-01528-2}
		
		\bibitem{mcconnell-robson}
		J. C. McConnell and J. C. Robson, \textit{Noncommutative Noetherian Rings}, revised ed., Graduate Studies in Mathematics, Vol.~30, American Mathematical Society, Providence, RI, 2001. \url{https://doi.org/10.1090/gsm/030}
		
		\bibitem{minac-palaisti-pasini-tan}
		J. Min\'a\v{c}, M. Palaisti, F. W. Pasini, and N. D. T\^an, Enhanced Koszul properties in Galois cohomology, \textit{Research in the Mathematical Sciences} \textbf{7} (2020), no.~2, Paper No.~10, 34 pp. \url{https://doi.org/10.1007/s40687-020-00208-5}
		
		\bibitem{minac-pasini-quadrelli-tan}
		J. Min\'a\v{c}, F. W. Pasini, C. Quadrelli, and N. D. T\^an, Koszul algebras and quadratic duals in Galois cohomology, \textit{Advances in Mathematics} \textbf{380} (2021), Article~107569. \url{https://doi.org/10.1016/j.aim.2021.107569}
		
		\bibitem{neukirch-schmidt-wingberg}
		J. Neukirch, A. Schmidt, and K. Wingberg, \textit{Cohomology of Number Fields}, 2nd ed., Grundlehren der mathematischen Wissenschaften, Vol.~323, Springer, Berlin, 2008. \url{https://doi.org/10.1007/978-3-540-37889-1}
		
		\bibitem{polishchuk-positselski}
		A. Polishchuk and L. Positselski, \textit{Quadratic Algebras}, University Lecture Series, Vol.~37, American Mathematical Society, Providence, RI, 2005. \url{https://doi.org/10.1090/ulect/037}
		
		\bibitem{priddy}
		S. B. Priddy, Koszul resolutions, \textit{Transactions of the American Mathematical Society} \textbf{152} (1970), 39--60. \url{https://doi.org/10.1090/S0002-9947-1970-0265437-8}
		
		\bibitem{quadrelli-few-relations}
		C. Quadrelli,
		Pro-$p$ groups with few relations and universal Koszulity, \textit{Mathematica Scandinavica} \textbf{127} (2021), no.~1, 28--42. \url{https://doi.org/10.7146/math.scand.a-123644}
		
		\bibitem{rusin}
		D. J. Rusin,
		The mod-$2$ cohomology of metacyclic $2$-groups, \textit{Journal of Pure and Applied Algebra} \textbf{44} (1987), 315--327. \url{https://doi.org/10.1016/0022-4049(87)90035-1}
		
		\bibitem{serre-galois}
		J.-P. Serre, \textit{Galois Cohomology}, Springer Monographs in Mathematics, Springer, Berlin, 1997. \url{https://doi.org/10.1007/978-3-642-59141-9}
		
		\bibitem{wall}
		C. T. C. Wall, Resolutions for extensions of groups, \textit{Proceedings of the Cambridge Philosophical Society} \textbf{57} (1961), no.~2, 251--255. \url{https://doi.org/10.1017/S0305004100035155}
		
		\bibitem{weibel}
		C. A. Weibel, \textit{An Introduction to Homological Algebra}, Cambridge Studies in Advanced Mathematics, Vol.~38, Cambridge University Press, Cambridge, 1994. \url{https://doi.org/10.1017/CBO9781139644136}
		
	\end{thebibliography}
\end{document}